\documentclass[10pt]{amsart}

\usepackage{mathtools, amsmath, amssymb}
\usepackage{hyperref}
\usepackage[all]{xy}
\usepackage{color,graphicx}
\usepackage{enumitem}
\usepackage{tikz, pgfplots}

\makeatletter
\renewcommand*\env@matrix[1][*\c@MaxMatrixCols c]{%
  \hskip -\arraycolsep
  \let\@ifnextchar\new@ifnextchar
  \array{#1}}
\makeatother

\newcommand{\on}[1]{\operatorname{#1}}
\newcommand{\bb}[1]{{\mathbb{#1}}}

\newcommand*\isom{\xrightarrow{\sim}}
\newcommand{\rr}{\bb{R}}
\newcommand{\pair}[1]{\langle#1\rangle}
\newcommand{\divisor}{\operatorname{div}}
\newcommand{\ord}{\operatorname{ord}}

\newcommand{\Car}{\operatorname{Car}}

\newcommand{\abs}[1]{\lvert#1\rvert}

\newcommand{\sing}{^{\textrm{sing}}}
\newcommand{\red}{{\textrm{red}}}

\newcommand{\ov}[1]{\overline{#1}}

\newcommand{\subpol}{{\rm sp}} 
\newcommand{\ddb}{\partial\bar\partial}
\newcommand{\db}{\bar\partial}

\def\qq{\mathbb{Q}}

\def\rr{\mathbb{R}}
\def\zz{\mathbb{Z}}
\def\cc{\mathbb{C}}

\def\mm{\mathcal{M}}
\def\BB{\mathcal{B}}

\def\deldelbar{\partial \bar{\partial}}
\def\d{\mathrm{d}}

\def\Gr{\mathrm{Gr}}

\def\omar{\omega_{\mathrm{Ar}}}

\theoremstyle{definition}
\newtheorem{definition}{Definition}[section]

\theoremstyle{plain}
\newtheorem{proposition}[definition]{Proposition}
\newtheorem{lemma}[definition]{Lemma}

\newtheorem{theorem}[definition]{Theorem}
\newtheorem{corollary}[definition]{Corollary}

\theoremstyle{remark}
\newtheorem{remark}[definition]{Remark}

\numberwithin{equation}{section}

\author{Jos\'e Ignacio Burgos Gil}
\address{Instituto de Ciencias Matem\'aticas (CSIC-UAM-UCM-UCM3).
  Calle Nicol\'as Ca\-bre\-ra~15, Campus UAM, Cantoblanco, 28049 Madrid,
  Spain.} 
\email{burgos@icmat.es}
\thanks{The first author has been partially supported by the MINECO research projects
  MTM2013-42135-P and MTM2016-79400-P, the ICMAT Severo Ochoa project
  SEV-2015-0554 and
  the DFG project SFB 1085 ``Higher Invariants''. The second and third authors would like to thank the  Instituto de Ciencias Matem\'eticas (ICMAT) for generously hosting us while much of this research was carried out. }

\author{David Holmes}
\address{Mathematical Institute,  
Leiden University,
PO Box 9512, 
2300 RA Leiden, 
The Netherlands}
\email{holmesdst@math.leidenuniv.nl}

\author{Robin de Jong}
\address{Mathematical Institute,  
Leiden University,
PO Box 9512, 
2300 RA Leiden, 
The Netherlands}
\email{rdejong@math.leidenuniv.nl}

\newcounter{nootje}
\newcommand{\mat}[1]{\left(\begin{matrix}#1\end{matrix}\right)}

\begin{document}

\title[Positivity of the height jump divisor]{Positivity of the
  height jump divisor}

\begin{abstract}  We study the degeneration of semipositive smooth hermitian line bundles on open complex manifolds, 
assuming that the metric extends well away from a codimension two analytic subset of the boundary. Using terminology introduced by R. Hain, we show that under these assumptions the so-called height jump divisors are always effective. This result is of particular interest in the context of biextension line bundles on Griffiths intermediate jacobian fibrations of polarized variations of Hodge structure of weight $-1$, pulled back along normal function sections. In the case of the normal function on $\mm_g$ associated to the Ceresa cycle, our result proves a conjecture of Hain. As an application of our result we obtain that the Moriwaki divisor on $\ov \mm_g$ has non-negative degree on all complete curves in $\ov \mm_g$ not entirely contained in the locus of irreducible singular curves.
\end{abstract}

\maketitle

\section{Introduction}

Let $X$ be a complex manifold, $E$ a reduced divisor on $X$, and
set $U=X\setminus |E|$.
Let $L$ be a holomorphic line bundle on $U$ equipped with a smooth
hermitian metric $\|\cdot\|$. The purpose of this note is to show a
certain concavity property for the singularities of the metric
$\|\cdot \|$ across the boundary divisor $E$, provided (a) the metric is
semipositive over $U$, and (b) the metric extends continuously away from an
analytic subset $S\subset |E|$  of
codimension at least $2$ in $X$. Here we say that the metric $\|\cdot\|$ is
semipositive if the first Chern form $c_1(\ov L)$ of $\ov L=(L,\|\cdot\|)$ is
semi-positive as a $(1,1)$-form on $U$. A typical
example of an analytic set $S$ as in condition (b) would be the
singular locus of $|E|$. 

To formulate the concavity property it is useful to introduce the notion of a Lear extension, following 
R. Hain \cite[Sections~6 and~14]{hain_normal}, inspired by the results of D. Lear in 
\cite{lear}.  We say that the smooth hermitian line bundle $\ov L$ is \emph{Lear-extendable}
over $X$ if there exists an analytic subset $S \subset |E|$ of
codimension at least $2$ in $X$ and an integer $N \in \zz_{>0}$ such that the smooth hermitian line bundle
$\ov L^{\otimes N}$ extends as a continuously metrized holomorphic line
bundle over $X \setminus S$. If we assume that the metric $\|\cdot\|$ is
semipositive in $U$, then  the assumption that $S$ is an analytic subset of
codimension at least $2$ in $X$  implies
that the holomorphic line bundle $L^{\otimes N}$  extends (uniquely) as a holomorphic line bundle on the
whole of $X$, by \cite[Theorem 1]{schiffman:eplb} and \cite[Proposition
2]{schumacher:ethlb}. 

Thus, if $\ov L$ is Lear-extendable over $X$ and the
metric $\|\cdot\|$ is semipositive on $U$, then the line bundle
$L$ has a canonical extension as a $\qq$-line bundle over
$X$, determined by the chosen metric $\|\cdot\|$ on $L$. We denote this canonical extension by $[\ov L,X]$
and call it the \emph{Lear extension} of $\ov L$. 

Let $C$ be a connected Riemann surface and let $\varphi \colon C \to X$ be a
holomorphic map such that the generic point of $C$ has image in $U$. Let $V
=\varphi^{-1}U$, which is an open dense subset of $C$. If we
assume that $\ov L$ is Lear-extendable over $X$ and in addition that $\varphi|_V^*\ov L$ is Lear-extendable over
$C$, it is natural to consider the difference of $\qq$-line bundles
\[  \varphi^* [\ov L,X] - [\varphi|_V^*\ov L,C] \]
on $C$. Since $\varphi^* [\ov L,X]$ and $[\varphi|_V^*\ov L,C]$ agree on $V$, the difference $\varphi^* [\ov L,X] - [\varphi|_V^*\ov L,C]$ determines canonically a $\qq$-divisor supported on
the boundary set $C \setminus V$, which we write as $J = J_{\varphi,\ov L}$. Following terminology introduced by Hain \cite[Section~14]{hain_normal} 
we call $J$ the \emph{height jump divisor} of $\ov L$ with respect to  
the morphism $\varphi$. We say that \emph{the height jumps} if $J$ is
non-zero. Note that the height jump divisor $J$ measures the failure of the
Lear extension to be compatible with pullback along $\varphi$.  

Our main aim in this note is to prove the following result. 
\begin{theorem} \label{main_intro} 
Assume that the metric of the smooth hermitian holomorphic line bundle $\ov L$ on $U$ is
semipositive. Moreover, assume that $\ov L$ has a Lear extension $[\ov L,X]$ over $X$, and that
$\varphi|_V^*\ov L$ has a Lear extension $[\varphi|_V^*\ov L,C]$ over
$C$.  Then the height jump divisor $J_{\varphi,\ov L}$ on $C$ is
effective. 
\end{theorem}

The main motivation for this result comes from work of Hain \cite{hain_normal}, G. Pearlstein \cite{pearldiff} and G. Pearlstein and C. Peters \cite{pearlpeters} concerning the  curvature properties
and asymptotic behavior of admissible biextension variations of mixed
Hodge structure, in particular a conjecture of Hain formulated in \cite[Section~14]{hain_normal}.

We briefly sketch this circle of ideas, referring to \cite{hain_normal} for more details.
Let $\bb{V}$ denote a polarized variation of Hodge
structure of weight $-1$ over the complex manifold $U=X\setminus |E|$. 
By classical work of P. Griffiths  one has associated  to $\bb{V}$ a canonical torus bundle $J(\bb{V}) \to U$, called the
intermediate jacobian fibration of $\bb{V}$ over $U$. As is explained in \cite{hainbiext}, \cite{hain_normal}
the polarization on $\bb{V}$ gives rise to a canonical holomorphic line bundle
$\BB$ on $J(\bb{V})$ whose total space classifies self-dual biextension variations of mixed Hodge structures of $\bb{V}$. A self-dual biextension variation of mixed Hodge structures of $\bb{V}$ is a variation of mixed Hodge structures $\bb{B}$ over $U$ with weight graded quotients
\[ \Gr_0 \cong \zz(0) \, , \quad \Gr_{-1} \cong \bb{V} \, , \quad \Gr_{-2} \cong \zz(1) \]
such that the resulting extensions $W_0\bb{B}/\zz(1) \in 
\mathrm{Ext}^1_{\mathrm{MHS}(U)}(\zz(0),\bb{V})$ and $W_{-1}\bb{B} \in \mathrm{Ext}^1_{\mathrm{MHS}(U)}(\bb{V},\zz(1))$ are identified via the canonical map 
\[  \mathrm{Ext}^1_{\mathrm{MHS}(U)}(\zz(0),\bb{V}) =J(\bb{V})(U) \longrightarrow \check{J}(\bb{V})(U)=\mathrm{Ext}^1_{\mathrm{MHS}(U)}(\bb{V},\zz(1)) \]
determined by the polarization of $\bb{V}$. The biextension line bundle $\BB$ is equipped with a
canonical smooth hermitian metric $\|\cdot\|$, and we write $\ov \BB = 
(\BB,\|\cdot\|)$.

Let $\nu \colon U \to J(\bb{V})$ be a holomorphic section of the torus bundle $J(\bb{V}) \to U$. Then any nowhere vanishing
holomorphic section $s$ of $\nu^*\BB$ can be seen as a self-dual
biextension variation of mixed Hodge structure of $\bb{V}$ on $U$.  We call 
\[ h = -\log \| s \| \]
the \emph{height} of (the biextension variation) $s$. Its Levi form
\[ \frac{i}{\pi} \,\deldelbar \,h \]
coincides with the first Chern form $c_1(\nu^*\ov \BB)$ of the smooth hermitian line bundle $\nu^*\ov \BB$ on $U$. 

We have the following two important results.
\begin{theorem} \label{semipositive} (Hain \cite[Theorem~13.1]{hain_normal}, Pearlstein and Peters \cite[Theorem~8.2]{pearlpeters}, C.~Schnell \cite[Lemma~2.2]{sc})
The first Chern form  $c_1(\nu^*\ov \BB)$ is semipositive on $U$.
\end{theorem}
\begin{theorem} \label{lear_exists} (Lear \cite[Theorem~6.6]{lear}, Hain \cite[Theorem~6.1]{hain_normal}, Pearlstein \cite[Theorem~5.19]{pearldiff}) 
Assume that $E$ is a reduced normal crossings divisor, that the variation of Hodge structure $\bb{V}$ is admissible, and that $\nu$ is an admissible normal function. Then the smooth hermitian line bundle $\nu^*\ov \BB$ is Lear-extendable over $X$. In fact, a positive tensor power of $\nu^*\ov \BB$ has a continuous extension over $X \setminus |E|^{\mathrm{sing}}$.
 \end{theorem}
 \newcommand{\var}{T}
We refer to \cite{pearldiff} for a discussion of admissibility of variations of mixed Hodge structures on $U$, and of normal functions. The admissibility condition on $\bb{V}$ implies that the monodromy operators of $\bb{V}$ about the irreducible components of $E$ are unipotent. A polarized variation of Hodge structures $\bb{V}$ having unipotent monodromy and induced by the cohomology of a projective smooth morphism of algebraic varieties $\var \to U$, and with $X \supset U$ a smooth algebraic variety with $X \setminus U$ a normal crossings divisor, is admissible. In the latter situation, any section $\nu$ of $J(\bb{V}) \to U$ determined by a family of cycles on $\var$ over $U$ via Griffiths's Abel-Jacobi map is an admissible normal function. 

Now as above let $\varphi \colon C \to X$ with $C$ a connected Riemann surface be a
holomorphic map such that the generic point of $C$ has image in $U$, and write $V
=\varphi^{-1}U$.  By combining Theorems \ref{main_intro}, \ref{semipositive} and \ref{lear_exists} we immediately obtain the following result.
\begin{theorem} \label{hain_conj} Assume that $E$ is a reduced normal crossings divisor, that the variation $\bb{V}$ on $U$ is admissible, and that the section $\nu$ is an admissible normal function. Then the height jump divisor $J_{\varphi,\nu^*\ov
  \BB}$ on $C$ determined by the smooth hermitian line bundle $\nu^*\ov \BB$ on $U$ and the map $\varphi$ is effective. 
  \end{theorem}

A special case of this result related to the so-called Ceresa cycle in the jacobian of a pointed compact connected Riemann surface of genus $g \geq 2$ was conjectured by Hain \cite[Conjecture~14.5]{hain_normal}. We will elaborate on this special case in Section \ref{sec:applications} below. 

For the special case of sections of families of principally polarized abelian varieties Theorem \ref{hain_conj} was shown, among other things, in our previous paper \cite{bhdj}. In this setting $\BB$ is the classical Poincar\'e bundle. The method of proof in \cite{bhdj} is based on an explicit formula for the metric on $\BB$ together with a precise analysis of its asymptotics. In the even more special case of sections of jacobians, the positivity of the height jump divisor was obtained in \cite{bihdj}, \cite{hdj}. In these two references, the positivity was obtained by establishing the concavity, in a suitable sense, of Green's functions of electric circuits. Note that none of these previous results were sufficient to treat the Ceresa cycle. 

A general interpretation of the height jump phenomenon for admissible biextension variations of mixed Hodge structures has been given by P. Brosnan and G. Pearlstein in terms of a pairing on the intersection cohomology of the underlying pure Hodge structure in \cite{brospearl}. This approach allows for an alternative proof of Theorem \ref{hain_conj}, cf. \cite[Corollary~11]{brospearl}.

We will actually prove a slight extension of Theorem \ref{main_intro} to the case of so-called subpolynomial Lear extensions, cf. Theorem \ref{thm:1} below. Moreover we will work with the slightly more flexible notion of metrized $\bb{R}$-divisors. For example, from a smooth hermitian holomorphic line bundle $\ov L$ we obtain a metrized $\bb{R}$-divisor by taking the divisor and height of a non-zero rational section. We define the subpolynomial Lear extension of a metrized $\bb{R}$-divisor in Section \ref{sec:metrized-r-divisors}. 

The structure of this note is as follows. In Section \ref{sec:prelim}
we introduce terminology and state some preliminary results, including
the key Lemma \ref{lemm:1} which is merely a local version of a lemma
from \cite[Chapter XI]{ACH}.   In Section \ref{sec:main} we state and
prove our main result, from which Theorem \ref{main_intro} immediately
follows. Finally in Section~\ref{sec:applications} we discuss
applications of our main result. We discuss the special case related
to the Ceresa cycle, as well as its ramifications for refined slope
inequalities on $\ov \mm_g$, following
\cite[Section~14]{hain_normal}. Also we discuss an example related to
Arakelov line bundles on families of compact and connected Riemann
surfaces, and the concavity of the local height jump multiplicities.  

\section{Preliminary results} \label{sec:prelim}

\subsection{Metrized $\bb{R}$-divisors}
\label{sec:metrized-r-divisors}

Our purpose is to study extensions of line bundles (or Cartier divisors)
determined by a smooth hermitian metric. Since a metric is an analytic object it seems
better to allow real coefficients for such extensions and not to be
confined to rational coefficients. Now for real coefficients the language
of Cartier divisors and Green functions is more natural than that of
line bundles and metrics, and thus we adapt this language from now on. 
For the theory of divisors on complex
manifolds the reader is referred to \cite{huybrechts:cg}.
In this section, we recall from
\cite{BurgosMoriwakiPhilipponSombra:aptv} the basic definitions
concerning 
metrized $\rr$-divisors. We
start by recalling the geometric theory of $\rr$-divisors, details of which can
be found in 
\cite{Lazarsfeld:posit_I}. 

Let $X$ be a complex manifold.  A Weil divisor on $X$ is a locally finite $\zz$-linear combination of irreducible 
analytic hypersurfaces of $X$. An analytic subset of $X$ is called irreducible if its set of non-singular points is connected. 
Since $X$ is smooth, the concepts of Cartier
and Weil divisor coincide. We denote by $\Car(X)$ the group of
Cartier divisors of $X$. The vector space of {$\rr$-Cartier divisors} of $X$ is
defined as
\begin{displaymath}
\Car(X)_\rr=\Car(X)\otimes_\zz \rr \, .
\end{displaymath}
Thus, an $\rr$-Cartier divisor on $X$ is a formal linear combination
$\sum_{i}\alpha_{i}D_{i}$ with $\alpha _{i}\in \rr$ and $D_{i}\in
\Car(X)$.  From now on $\rr$-Cartier divisors will be called
$\rr$-divisors, while Cartier divisors will be called divisors. The
support of a divisor $D$ is denoted by $|D|$.
An $\rr$-divisor $D$ is said to be \emph{ample} (respectively \emph{effective},
\emph{nef}) if $D$ is an $\bb{R}$-linear combination of ample (respectively
effective, nef) divisors with positive coefficients.

We now introduce metrized divisors and metrized
$\rr$-divisors. 
\begin{definition}\label{def:2}
  Let $D$ be a  Cartier divisor on $X$. A \emph{Green function} for $D$ is a
  smooth function $g\colon X\setminus |D|\to \bb{R}$ such that, for each
  point $p\in X$, there is a neigbourhood $V$ of $p$ and a local
  equation $f$ of $D|_V$ such that $g+\log|f|$ extends to a smooth
  function on $V$. 
  A \emph{metrized divisor} on $X$ is a pair $\overline D=(D,g_{D})$ where $D$ is a
  divisor on $X$ and $g_D$ 
  is a Green function for $D$. The group of metrized divisors on $X$ is denoted by
  $\widehat \Car(X)$.
\end{definition}

\begin{remark}\label{rem:1} Given a holomorphic line bundle $L$ on $X$ provided
  with a smooth hermitian metric $\|\cdot\|$ and a non-zero
  meromorphic section $s$,  
  we obtain a metrized divisor
  \begin{displaymath}
    \widehat \divisor (s)=(\divisor(s),-\log\|s\|) 
  \end{displaymath}
  on $X$.
  However we note that the notions of hermitian line bundle with section and of metrized
  divisor are not exactly the same. For example, if we multiply the section $s$ by
  a non-zero complex number $\zeta$ and the metric by the positive
  real number $|\zeta|^{-1}$, the associated metrized divisor does not change.  
\end{remark}

\begin{definition} \label{def:1} Let $X$ be a complex manifold.  The
  \emph{group of metrized $\rr$-divisors} on $X$ is the
  quotient
  \begin{displaymath}
    \widehat\Car(X)_{\rr}=\left. \widehat \Car(X)\otimes_\zz \rr \right/ \sim
  \end{displaymath}
  where $\sim$ is the equivalence relation given by $ \sum_{i}\alpha
  _{i}\ov D_{i}\sim \sum_{j}\beta _{j}\ov E_{j} $ if and only if $
  \sum_{i}\alpha _{i} D_{i}= \sum_{j}\beta _{j} E_{j} $ and
  there is a dense open subset
  $U$ of $ X$ such that
  \begin{displaymath}
      \sum_{i}\alpha _{i} g_{\ov D_{i}}(p)= \sum_{j}\beta _{j}g_{\ov
        E_{j}} (p)\quad \text{ for all} \, p\in U.
  \end{displaymath}
  Let $D=\sum_i \alpha_i D_i$. Then the function
  $g_{D}=\sum_{i}\alpha _{i} g_{\ov D_{i}}\colon U\to \rr$ is called
  the \emph{Green function} of $\ov D$.
\end{definition}

\begin{definition}\label{def:3} 
The first Chern form of a metrized
  $\bb{R}$-divisor $\ov D=(D,g_D)$ is
  \begin{displaymath}
    c_{1}(\overline D)=\frac{i}{\pi}\,\partial\bar \partial \, g_{D}.
  \end{displaymath}
  This form can be extended to a smooth $(1,1)$-form on  the whole $X$. 
\end{definition}

\begin{remark} For a holomorphic line bundle $L$ on $X$ with smooth
  hermitian metric $\|\cdot\|$ and non-zero
meromorphic section $s$ we have $c_1(\widehat \divisor (s))=c_1(L,\|\cdot\|)$.
\end{remark}

\subsection{Subpolynomial Lear extensions}

In this section we introduce subpolynomial Lear extensions of metrized $\bb{R}$-divisors. This
concept is a mild generalization of the Mumford-Lear extensions introduced
in the context of metrized line bundles in \cite{bkk} and is suggested by \cite[Lemma XI.9.17]{ACH} where it
is proved that, for compact Riemann surfaces, semipositivity and subpolynomial growth
together are enough to compute the degree of a hermitian line bundle using Chern-Weil theory. We will crucially 
use a local version of \cite[Lemma XI.9.17]{ACH} in the proof of our main result.

\begin{definition}\label{def:4} A function $\mu \colon \rr_{>0}\to
  \rr_{>0}$ is called \emph{subpolynomial} if, for all $n\in \rr_{>0}$,
  \begin{equation*}
    \lim_{x \to \infty} \frac{\mu (x)}{x^n} = 0\, . 
  \end{equation*}
\end{definition}

We will use subpolynomial functions to control the growth of our
metrics along the boundary. A typical example of a subpolynomial
function is the logarithm $\mu(x)=\log(1+x)$. 

Let $X$ be a complex manifold of pure dimension $d$, and $E$ a reduced
divisor on $X$.  Write $U = X \setminus |E|$.

\begin{definition}\label{def:5} Let $\ov D=(D,g_{D})$ be a metrized
  $\rr$-divisor on $U$. We say that a \emph{subpolynomial} (\subpol-) \emph{Lear extension} of
  $\ov D$ over 
  $X$ exists if there is an auxiliary metrized $\rr$-divisor $\ov
D_{X}=(D_{X},g_{D_{X}})$ on $X$ with $D_{X}|_{U}=D$ and an analytic
subset $S\subset |E|$
  of codimension at least $2$ in $X$, such that for all $p\in |E|\setminus
  S$ there is an 
  open neighbourhood $V$ of $p$ in $X\setminus S$, a local equation $f$ of $E$, a
  real number $a$ and a subpolynomial function $\mu $ such that the
  inequalities  
\begin{equation}\label{eq:6}
  -\log(\mu(1/|f|))\le g_{D}-g_{D_X}+a\log|f| \le  \log(\mu(1/|f|))
\end{equation}
are satisfied on $V\cap U$. 
\end{definition}

\begin{remark}\label{rem:2} The real
  number $a$ in the previous 
  definition depends on the choice of the auxiliary metrized divisor $\ov
 D_{X}$ and on the point $p$. However by definition the real number $a$ is locally
  constant on $|E|\setminus S$. In particular, if $p,q$ belong to the same
  irreducible component of $E$, then the real numbers associated  to $p,q$ agree. 
\end{remark}

\begin{definition}\label{def:6}
  Let $\ov D=(D,g_{D})$ be a metrized $\rr$-divisor on $U$.
  Let $\{E_{i}\}_{i\in I}$ be the irreducible components of $E$.
  Assume that a \subpol-Lear extension of $\ov D$ over $X$
  exists, with auxiliary metrized divisor $\ov D_X$. By Remark \ref{rem:2} one has well defined real numbers 
  $a_{i}$, $i\in I$, one for each irreducible component  $E_i$ of $E$. We define
  the \emph{\subpol-Lear extension of $\ov D$ over $X$} as the $\bb{R}$-divisor
  \begin{equation}\label{eq:4}
    [\ov D,X]=D_{X}+\sum _{i\in I}a_{i}E_{i}\, .
  \end{equation}
Note that the sum in \eqref{eq:4} is locally finite, as $E$ is a
divisor. 
\end{definition}

\begin{proposition}\label{prop:1} If $\ov D$ has a \subpol-Lear
  extension over $X$, it is unique. 
\end{proposition}
\begin{proof}
  Let $\ov D'_{X}=(D'_{X},g_{D'_{X}})$ and $S'$ be a different choice of
  auxiliary metrized $\rr$-divisor and analytic subset in Definition
  \ref{def:5}. Then there are real numbers
  $\{b_{i}\}_{i\in I}$ such
  that
  \begin{displaymath}
    D'_{X}=D_{X}+\sum_{i\in I}b_{i}E_{i} \, .
  \end{displaymath}
  Without loss of generality we may assume that
  $S=S'$ and that $|E|^{\rm sing}\subset S$.  
  Let $p\in |E|\setminus S$ and $V$ a small enough neighborhood of $p$
  that only intersects a single component $E_{i}$ of $E$. Let
  $f'$ be another choice of local equation for $E$ on $V$. Then
  $f'=uf$ for a non-vanishing holomorphic function $u$ on $V$. Finally
  let $\mu '$ be another choice of subpolynomial function such that
  \begin{equation}\label{eq:5}
    -\log(\mu'(1/|f'|))\le g_{D}-g_{D'_X}+a_{i}'\log|f'| \le  \log(\mu'(1/|f'|))\, . 
  \end{equation}
  After shrinking $V$ if necessary, we can find a third subpolynomial
  function $\mu ''$ such that
  \begin{displaymath}
    0< \log(\mu'(1/|f'|))\le \log(\mu''(1/|f|)) \quad\text{and}\quad
    0<\log(\mu(1/|f|))\le \log(\mu''(1/|f|)) \, .
  \end{displaymath}
  By definition of Green function, there is a smooth function
  $\varphi$ on $V$ such that 
  \begin{displaymath}
    g_{D'_{X}}=g_{D_{X}}-b_{i}\log|f|+\varphi \, .
  \end{displaymath}
  Then, subtracting equation \eqref{eq:6} from equation \eqref{eq:5}
  we deduce
  \begin{displaymath}
      -2\log(\mu''(1/|f|))\le (b_{i}+a_{i}'-a_{i})\log |f|+\log|u|-\varphi \le  2 \log(\mu''(1/|f|))
  \end{displaymath}
  which implies that $a_{i}=a_{i}'+b_{i}$. We conclude
  \begin{displaymath}
    D'_{X}+\sum_{i\in I} a'_{i}E_{i}=D_{X}+\sum_{i\in I}
    (b_{i}+a'_{i})E_{i}= D_{X}+\sum_{i\in I} a_{i}E_{i} \, ,
  \end{displaymath}
  and this proves the result.
\end{proof}

\subsection{Residues of semipositive forms}

We will write $\Delta \subset \cc$ for the closed unit disk, and $\Delta
^{\ast}=\Delta \setminus \{0\}$ for the punctured closed disk.
For  $0<t \leq 1$, we denote by $\Delta _{t}$ the closed
disk of radius $t$.

Let $U$ be a complex manifold and $\omega$ a $(1,1)$-form on $U$. Recall that
$\omega$ is called \emph{semipositive} if for every continuous map 
$\varphi\colon \Delta 
  _{t }\to U$  which is holomorphic on the interior of $\Delta _{t}$, the positivity condition 
  \begin{displaymath}
    \int_{\Delta _{t}}\varphi^{\ast}\omega\ge 0
  \end{displaymath}
holds. 
Fix $0<s\le 1$.
 Let $g$ be a smooth function on an open neighbourhood of $\Delta_{s}^*$
 taking values in $\rr$. We write
\begin{equation*}
\omega \coloneqq \frac{i}{\pi}\ddb g \, , \quad \eta \coloneqq
\frac{i}{\pi }\bar\partial g \, . 
\end{equation*}
The next lemma is a local variant on \cite[Lemma XI.9.17]{ACH}, where a
global version is given. We give the direct proof of the local version
here, following the arguments in \cite{ACH}, since it only takes little
longer than deducing the local case from the global case. Moreover
the result is central in our arguments.   
\begin{lemma} \label{lemm:1}
Suppose that $\omega$ is a semipositive $(1,1)$-form on an open
neighbourhood of $\Delta_{s}^*$, and that there exist $a \in \bb R$ and a
subpolynomial function $\mu $ such that  
\begin{equation}\label{eq:growth_bound_for_residue}
-\log(\mu(1/|z|))\le g(z)+a\log|z| \le \log(\mu(1/|z|)) \, .
\end{equation}
Then $\int_{\Delta_{s}} \omega\coloneqq \lim_{\varepsilon \to 0}\int_{\Delta_{s}
  \setminus \Delta_{\varepsilon}} \omega$ exists in $\bb R$, and
we have  
\begin{equation*}
0 \leq \int_{\Delta_{s}} \omega = \int_{\partial \Delta_{s}} \eta + a \, . 
\end{equation*}
\end{lemma}
\begin{proof}
For ease of notation we will treat only the case $s=1$. 
For each $0 < t\le 1$, a computation in polar coordinates yields
\begin{equation*}
\bar\partial g|_{\partial \Delta_t} = \left( -
  \frac{i}{2}r\frac{\partial g}{\partial r} -
  \frac{1}{2}\frac{\partial g}{\partial \theta}\right) \on d \theta. 
\end{equation*}
We define
\begin{equation*}
F(t) = - \int_{\partial \Delta_t} \eta  = \frac{-1}{2\pi}
\int_{\partial \Delta_t} r \frac{\partial g}{\partial r} \on d \theta.  
\end{equation*}
For $t_{1}<t_{2}$ we have using Stokes's theorem
\begin{equation*}
  F(t_{2})-F(t_{1})= - \int_{\partial \Delta_{t_{2}}} \eta+
  \int_{\partial \Delta_{t_{1}}} \eta =-\int_{\Delta_{t_{2}}\setminus
  \Delta_{t_{1}}}\omega \le 0.
\end{equation*}
Thus, the function $F$ is non-increasing.

 Now for any $0 < \varepsilon \le t \le 1$ we compute
\begin{equation*}
\begin{split}
F(t) (\log t - \log \varepsilon) & = \int_\varepsilon^t F(t) \frac{\on d r}{r}\\
& \le \int_\varepsilon^t F(r) \frac{\on d r}{r}\\
& = \frac{-1}{2\pi} \int_\varepsilon^ t \int_0^{2 \pi} \frac{\partial g}{\partial r} \on d r \on d \theta\\
& = \frac{-1}{2\pi}\int_0^{2 \pi} g(te^{i \theta}) - g(\varepsilon e^{i \theta}) \on d \theta. 
\end{split}
\end{equation*}
Fixing $t$ and letting $\varepsilon$ vary we find
\begin{equation*}
-F(t) \log \varepsilon \le  \text{constant} + \frac{1}{2\pi} \int_0^{2 \pi} g(\varepsilon e^{i \theta}) \on d \theta. 
\end{equation*}
Applying the rightmost  bound from formula
\eqref{eq:growth_bound_for_residue} we see that
\begin{equation*}
  -F(t)\log \varepsilon \le \text{constant} - a \log(\varepsilon )+\log
(\mu(1/\varepsilon)).  
\end{equation*}
The fact that $\mu $ is subpolynomial and that, for $\varepsilon <1$,
$-\log \varepsilon > 0$,
implies that $F(t) \le a$. In
particular we find  that the limit
\begin{equation*}
K\coloneqq \lim_{t \to 0} - \int_{\partial \Delta_t} \eta  
\end{equation*}
exists and that  $K\le a$.

Using again that $F$ is non-increasing, we deduce that $-K\le -F(t)$
for all $t$. Thus, for $0<\varepsilon <t\leq1$,
\begin{displaymath}
  -K(\log t-\log \varepsilon )\le \int_\varepsilon^t -F(r) \frac{\on d
    r}{r} = \frac{1}{2\pi}\int_0^{2 \pi} g(te^{i \theta}) -
  g(\varepsilon e^{i \theta}) \on d \theta.  
\end{displaymath}
Fixing $t$ and applying the leftmost  bound from formula
\eqref{eq:growth_bound_for_residue} we obtain
\begin{displaymath}
  K\log \varepsilon \le \text{constant} + a \log(\varepsilon )+\log
(\mu(1/\varepsilon)) \, .
\end{displaymath}
From this we deduce that $K\ge a$, so that in fact $K=a$, and proving in
particular that 
\begin{equation} \label{limit_to_zero} \lim_{t \to 0} \int_{\partial \Delta_{t}} \eta + a = 0  \, . 
\end{equation}
Finally we apply Stokes's Theorem to deduce
\begin{equation*}
0\leq \int_\Delta \omega = \lim_{t \to 0} \int_{\Delta \setminus \Delta_t}
\omega = \int_{\partial \Delta} \eta - \lim_{t \to 0} \int_{\partial
  \Delta_t} \eta = \int_{\partial \Delta} \eta + a.
\end{equation*}
This proves the lemma.
\end{proof}

\begin{corollary}\label{cor:1}
  Let $R=\{z_{1},\dots,z_{k}\}\subset \Delta _{s}$ be a finite set of points with
  $|z_{i}|<s$, and let $g$ be a smooth function on an open neighborhood of
  $\Delta _{s}\setminus R$ such that $\omega =\frac{i}{\pi} \ddb g $
  is semipositive, and such that there is a subpolynomial function $\mu $ and real
  numbers $a_{i}$ such that 
  \begin{displaymath}
    -\log(\mu(1/|z-z_{i}|))\le g(z)+a_i\log|z-z_{i}| \le
    \log(\mu(1/|z-z_{i}|))
  \end{displaymath}
  when $z$ is close to $z_{i}$. Then the (in)equalities
  \begin{displaymath}
    0 \leq \int_{\Delta_{s}} \omega = \int_{\partial \Delta_{s}} \eta + \sum_{i=1}^{k}a_{i} 
  \end{displaymath}
hold, where $\eta =\frac{i}{\pi }\bar\partial g $.
\end{corollary}
\begin{proof}
  It is enough to cut  out a small disc around each point $z_{i}$, contained
  in the interior of $\Delta _{s}$, to apply
  Lemma \ref{lemm:1} to each of these discs, and to apply Stokes's theorem to the
  complement of the discs.
\end{proof}
A metrized $\rr$-divisor $\ov D=(D,g_D)$ on $U$ is called semipositive if $c_{1}(\ov
D)$ is a semipositive $(1,1)$-form on $U$.
\begin{corollary} \label{cor:2} Assume that $X$ is a Riemann surface. Let $\ov D$ be a metrized $\rr$-divisor on $U=X\setminus |E|$. Assume that $\ov D$ is semipositive and that $\ov D$ has a \subpol-Lear extension over $X$. Then
$c_1(\ov D)$ is locally integrable on $X$. If $X$ is compact, then the (in)equalities
\[ 0\leq \int_U c_1(\ov D) = \deg\,[\overline D,X] \]
hold.
\end{corollary}
\begin{proof} We write 
\begin{displaymath}
  \omega =\frac{i}{\pi }\ddb g_{D}\, , \qquad \eta= \frac{i}{\pi }\db g_{D}
\end{displaymath}
on $U$ as before.
The local integrability of $c_1(\ov D)=\omega$ follows directly from Lemma \ref{lemm:1}. Assume that $X$ is compact. The bound $\int_U c_1(\ov D) \ge 0$ is then clear. Let $p \in |E|$.  Applying the result found in equation (\ref{limit_to_zero})  to small disks $\Delta_\varepsilon$ around $p$ we find that
\[ -\lim_{\varepsilon \to 0} \int_{\partial \Delta_\varepsilon} \eta = a  \, , \]
where $a$ is the local contribution at $p$ to $\deg  \,[\overline D,X]$. The usual argument using Stokes's theorem then gives the stated equality.\end{proof}
\begin{remark} Let $U$ be a connected Riemann surface with smooth compactification $X$, and let $\bb{V}$ be an admissible polarized variation of Hodge structures of weight $-1$ as in the Introduction. Let $\nu \colon U \to J(\bb{V})$ be an admissible normal function and denote by $\ov \BB$ the canonically metrized biextension line bundle on $J(\bb{V})$. Using Theorems \ref{semipositive} and \ref{lear_exists}, as a special case of Corollary \ref{cor:2} we obtain the local integrability of the $(1,1)$-form $ c_1(\nu^*\ov \BB)$, together with the equality
\[ \int_U c_1( \nu^*\ov \BB) = \deg \,[\nu^*\ov \BB,X  ] \, . \]
This result proves \cite[Conjecture~6.6]{hain_normal}. That \cite[Conjecture~6.6]{hain_normal} holds has previously been noted by Pearlstein in \cite[Corollary 3.2]{pe_let} and by Pearlstein and Peters in \cite[Section~1.5.6]{pearlpeters}.
\end{remark}

 \section{Positivity of the height jump divisor}
\label{sec:main}

Let $X$ be a complex manifold, $E$ a reduced divisor on $X$
and $U=X\setminus |E|$.  Let $\ov D=(D,g_{D})$ be a metrized
$\rr$-divisor on $U$. Let $C$ be a connected Riemann surface and $\varphi\colon
C\to X$ a holomorphic map 
 such that the image of the generic point of $C$ lies in
$U\setminus |E|$. Write $V=\varphi^{-1}(U)$.  
We have a well defined $\rr$-divisor $\varphi|_V^* D$ on $V$ and a Green
function $\varphi|_V^*g_{D}=g_{D}\circ \varphi|_V$ giving a metrized $\rr$-divisor
$\varphi|_V^*\ov D=(\varphi|_V^* D,\varphi|_V^*g_{D})$ on $V$. 

\begin{definition} \label{def:7}
Let $X$, $U$, $\ov D$ and $\varphi \colon C \to X$ be as above. 
Assume that $\ov D$ admits a
$\subpol$-Lear extension $[\ov D,X]$ to $X$. Assume furthermore that
the metrized divisor 
$\varphi|_V^*\ov D$ on $V$ admits a $\subpol$-Lear
extension $[\varphi|_V^*\ov D, C]$ to $C$.
Then the
  \emph{height jump divisor} is defined as the difference
  \begin{displaymath}
    J_{\varphi,\ov D} = \varphi^*[\ov D, X] - [\varphi|_V^*\ov D, C] \, ,
  \end{displaymath}  
which is an $\bb R$-divisor on $C$. 
\end{definition}
In particular, the height jump divisor measures the lack of functoriality of a
\subpol-Lear extension. 
The main result of the present note is the
following result, of which Theorem \ref{main_intro} is an immediate consequence. Note that we
 are neither assuming $X$ nor $C$ to be compact, at this point.

\begin{theorem}\label{thm:1} Assume that the hypotheses of
  Definition \ref{def:7} are satisfied, and that the metrized divisor $\ov D = (D,g_{D})$
  on $U$ is semipositive, that is, the smooth $(1,1)$-form $c_{1}(\ov D)$
  on $U$ is semipositive. Then the height jump divisor $J_{\varphi,\ov D}=\varphi^*[\ov D, X]
  - [\varphi|_V^*\ov D, C]$ is an effective $\rr$-divisor on $C$. 
  \end{theorem}
\begin{proof}
  Since we are assuming that $\ov D$ has a \subpol-Lear extension over
  $X$, we have in particular an auxiliary metrized $\rr$-divisor $\ov D_X =(D_{X},g_{D_X})$ on $X$ such that
  $D_{X}|_{U}=D$. We may assume that $D_{X}$ has no common components
  with $E$. 
  Next, observe that the statement is local on $C$, and that the
  support of the height jump divisor is contained in $C\setminus
  V$. Let $p\in C$ such that $\varphi(p)\in |E|$. Choose a local
  coordinate $z$ in a neighborhood of $p$ such that $p$ is given by
  $z=0$, write 
  \begin{displaymath}
    B_{\varepsilon }(p)=\{q\in C\mid \abs{z(q)}\le \varepsilon \}
  \end{displaymath}
  and choose an $\varepsilon >0$ small enough such that
  $B_{\varepsilon }(p)\cap \varphi^{-1}(|E|\cup |D_X|)=\{p\}$  and that
  $B_{\varepsilon }(p)$ is homeomorphic to a closed disk. 
For $\delta \ge 0$ and small enough, let $\varphi_{\delta }\colon
  B_{\varepsilon }(p)\to X$ be a generic deformation of $\varphi$
  varying continuously with $\delta $. By this we mean that
  \begin{enumerate}
  \item for all $\delta$ the map $\varphi_{\delta}$ is continuous, and
    holomorphic on the interior of $B_{\varepsilon}(p)$;
    \item $\varphi_{0}=\varphi|_{B_{\varepsilon }(p)}$;
  \item for $\delta >0$, $\varphi_{\delta }(B_{\varepsilon }(p))$ meets
    $E\cup D_{X}$ properly and transversely and avoids $(|E|\cup |D_{X}|)^{\sing}$; 
  \item the number of intersection points of $\varphi_{\delta
    }(B_{\varepsilon }(p))$ with $|E|$ (respectively with $|D_{X}|$) is
    constant and agrees with the multiplicity at $p$ of $\varphi^*E$
    (respectively of  $\varphi^*D_{X,\red}$). 
  \end{enumerate}
A generic deformation can be easily constructed using the
Transversality Theorem \cite[Pag. 68]{GP}. Such a generic deformation is
represented in Figure \ref{fig:gd}.
  
  \begin{figure}
    \centering
      \begin{tikzpicture}
\begin{axis}[hide axis, x=2.5cm,y=2.5cm]
\addplot [domain=-1.4:1.4, samples=100, variable=\t] ({t^2-1 + t^3-t}, {t^2-1 - (t^3-t)} ); 
\addplot [domain=-100:0, samples=100, variable=\t] ({-2*cos(t) + 2*0.7071067812}, {-2*sin(t) - 2*0.7071067812} ); 
\addplot [domain=-150:-100, samples=100, dashed, variable=\t] ({cos(t) - 0.7071067812}, {sin(t) + 0.7071067812} );
\addplot [domain=0:50, samples=100, dashed, variable=\t] ({cos(t) - 0.7071067812}, {sin(t) + 0.7071067812} );
\addplot [domain=-100:0, samples=100, variable=\t] ({cos(t) - 0.7071}, {sin(t) + 0.7071} ); 
\addplot [domain=-100:0, samples=100, variable=\t] ({cos(t) - 0.95*0.7071}, {sin(t) + 0.95*0.7071} ); 
\addplot [domain=-100:0, samples=100, variable=\t] ({cos(t) - 0.9*0.7071}, {sin(t) + 0.9*0.7071} ); 
\addplot [domain=-100:0, samples=100, variable=\t] ({cos(t) - 0.85*0.7071}, {sin(t) + 0.85*0.7071} ); 
\addplot [domain=-100:0, samples=100, variable=\t] ({cos(t) - 0.8*0.7071}, {sin(t) + 0.8*0.7071} ); 
\addplot [domain=-100:0, samples=100, variable=\t] ({cos(t) - 0.75*0.7071}, {sin(t) + 0.75*0.7071} ); 
\node at (axis cs:2,-0.2) {$E$};
\node at (axis cs:1.6,0.7) {$D_X$};
\node at (axis cs:-1.4,0.2) {$\varphi_0$};
\addplot[] coordinates {(0.21,0.2)} node[pin={[pin distance = 1.5cm]160:{$\varphi_0(B_\epsilon(p))$}}]{} ;
\node at (axis cs:1.1,1.1) {$\varphi_\delta(B_\epsilon(p))$};
\addplot[] coordinates {(0.28,0.63)} node[pin={[pin distance = 1.45cm]40:{}}]{} ;
\addplot[] coordinates {(0.315,0.595)} node[pin={[pin distance = 1.45cm]40:{}}]{} ;
\addplot[] coordinates {(0.35,0.56)} node[pin={[pin distance = 1.6cm]40:{}}]{} ;
\addplot[] coordinates {(0.385,0.525)} node[pin={[pin distance = 1.75cm]40:{}}]{} ;
\addplot[] coordinates {(0.42,0.49)} node[pin={[pin distance = 1.9cm]40:{}}]{} ;
\end{axis}
\end{tikzpicture}
    \caption{Generic deformation}
    \label{fig:gd}
  \end{figure}
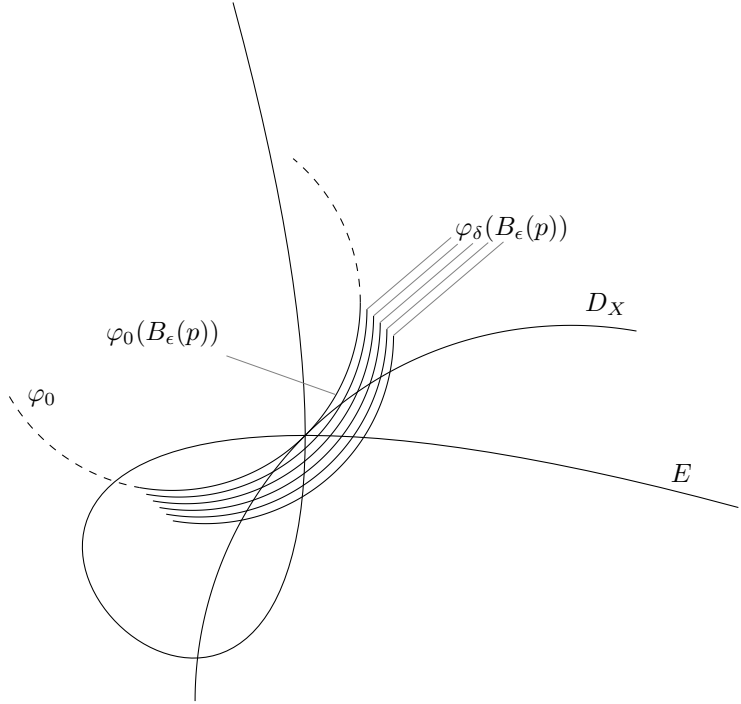
  
Let $E_{1},\dots,E_{k}$ be the irreducible components of $E$ that meet
a small neighborhood of $\varphi(p)$.
We denote by $m$ the multiplicity of $\varphi^{\ast}D_{X,\red}$ at the point
$p$ and by $m_{i}$ the multiplicity at $p$ of $\varphi^{\ast}E_{i}$,
$i=1,\dots,k$. Let $a_{1},\dots,a_{k}$ be the real numbers appearing
in Remark \ref{rem:2} for the \subpol-Lear extension of $\ov D$ to $X$
and the auxiliary metrized $\bb{R}$-divisor $\ov D_{X}$
and $a_{C}$ the one for the \subpol-Lear extension of
$\varphi^{\ast}|_V\ov D$ to $C$ at $p$ with auxiliary divisor
$\varphi^{\ast}\ov D_{X}$ . Then the multiplicity of the height jump  
divisor $J_{\varphi,\ov D}$ at $p$ is given by
\begin{displaymath}
  \sum_{i=1}^{k} m_{i}a_{i}+m-a_{C}-m=\sum_{i=1}^{k}
  m_{i}a_{i}-a_{C}. 
\end{displaymath}
Thus we are reduced to proving that $a_{C}\le \sum m_{i}a_{i}$. 
We write 
\begin{displaymath}
  \omega =\frac{i}{\pi }\ddb g_{D}\, , \qquad \eta= \frac{i}{\pi }\db g_{D}
\end{displaymath}
on $U$ as usual.
From Corollary \ref{cor:1} we deduce that
for all small $\delta >0$ we have
\begin{displaymath}
  0\le \int_{B_{\varepsilon }(p)}\varphi_{\delta }^{\ast}\omega =
  \int_{\partial B_{\varepsilon }(p)}\varphi_{\delta
  }^{\ast}\eta+\sum_{i=1}^{k}  m_{i}a_{i}+m \, . 
\end{displaymath}
Since the image of $\partial B_{\varepsilon }(p)$ by
$\varphi_{\delta}$ varies in a region where the form $\eta$ is smooth
we can take the limit when $\delta $ goes to zero in the previous
equation to deduce
\begin{equation}
  \int_{\partial B_{\varepsilon }(p)}\varphi^{\ast}\eta+\sum_{i=1}^{k}
  m_{i}a_{i}+m \ge 0 \, . \label{eq:1}
\end{equation}
From Lemma \ref{lemm:1} we obtain that
\begin{equation}
\label{eq:2}
  \lim_{\varepsilon \to 0} \int_{\partial B_{\varepsilon }(p)}\varphi^{\ast}\eta=-a_{C}-m.
\end{equation}
Combining equations \eqref{eq:1} and \eqref{eq:2} we deduce $a_{C}\le
\sum_{i=1}^k m_{i}a_{i}$ as required.
\end{proof}

\section{Applications} \label{sec:applications}

The purpose of this section is to discuss some applications of the results shown above.

\subsection{A conjecture of Hain}

Let $X$ be a complex manifold, $E$ a reduced normal crossings divisor on $X$, and write $U=X\setminus |E|$. Let $\bb{V}$ be an admissible polarized variation of Hodge structures of weight $-1$ on $U$, let $J(\bb{V})\to U$ denote the associated Griffiths intermediate jacobian fibration, and let  $\nu \colon U \to J(\bb{V})$ be an admissible normal function as in the introduction. Let $\ov \BB$ be the canonically metrized biextension line bundle on $J(\bb{V})$. As was already explained in the introduction, results of Lear, Hain, Pearlstein, Schnell and Pearlstein and Peters (cf. Theorems \ref{semipositive} and \ref{lear_exists}) show that the assumptions in Theorem \ref{main_intro} are satisfied for the smooth hermitian line bundle $\nu^*\ov \BB$ on $U$. Thus, let $C$ be a connected Riemann surface and $\varphi\colon C\to X$ a holomorphic map 
 such that the image of the generic point of $C$ lies in
$U\setminus |E|$, then we derive
\begin{theorem} \label{hain_conj_bis} 
The height jump divisor $J_{\varphi,\nu^*\ov
  \BB}$ on $C$ determined by  $\nu^*\ov \BB$ and the map $\varphi$ is effective. 
  \end{theorem}

Fix an integer $g \geq 2$. We will apply Theorem \ref{hain_conj_bis} to the case where $U=\mm_g$ is the moduli stack (orbifold) of smooth complete complex curves of genus $g$, and where as compactification $X \supset U$ we choose $X = \ov \mm_g$, the moduli stack of stable curves of genus $g$. We note that both $U$ and $X$ are stacks on the site of complex manifolds with \'etale covers in the sense of \cite[Tag 02ZH]{stacks-project}, rather than actual complex manifolds. Although our definitions and results above do not apply directly to this more general setting, they do carry over straightforwardly: first of all our stacks have representable diagonals, and admit an \'etale surjective map from a complex manifold. These two properties follow immediately from the analogous results in the algebraic setting, for which we refer to \cite{Deligne1969The-irreducibil}. Second, if a Lear extension of a smooth hermitian line bundle exists, it is unique, and moreover its formation commutes with arbitrary \'etale base change, for example by dimension considerations. For similar reasons, our discussion of intermediate jacobian fibrations, polarizations, biextensions and normal functions extends to the more general setting. We note that the boundary divisor $X \setminus U$ has normal crossings.

By associating to each moduli point $[Y]$ of $\mm_g$, where $Y$ is a smooth complete curve of genus $g$, its first homology group $H_1(Y,\zz(1))$ we obtain a tautological polarized admissible variation of Hodge structures $\mathbb{H}$ of weight $-1$ on $\mm_g$. Denote by $\mathbb{L}$ the variation of Hodge structure $(\bigwedge^3 \mathbb{H})(-1)$, then wedging with the polarization gives a canonical inclusion $\mathbb{H} \hookrightarrow \mathbb{L}$. Denote by $\bb{V}$ the quotient $\bb{L}/\bb{H}$. Then both $\bb{L}$ and $\bb{V}$ are admissible polarized variations of Hodge structure of weight $-1$ on $\mm_g$. As it turns out, the universal Ceresa cycle in the jacobian of a (pointed) complete curve of genus $g$ canonically gives rise, by the Griffiths Abel-Jacobi map, to an admissible normal function $\nu \colon \mm_g \to J(\bb{V})$. For this particular normal function, the result in Theorem \ref{hain_conj_bis} was conjectured by Hain, see \cite[Conjecture~14.5]{hain_normal}. The possibility that the more general result of Theorem \ref{hain_conj_bis} could be true was suggested by Hain at the end of \cite[Section~14]{hain_normal}.

\subsection{Positivity and slope inequalities} 

For a moment we return to the setting of Section~\ref{sec:main}. In particular  we have $X$  a complex manifold, $E$ a reduced divisor on $X$
and we set $U=X\setminus |E|$.  Let $\ov D=(D,g_{D})$ be a metrized
$\rr$-divisor on $U$, let $C$ be a connected Riemann surface and $\varphi\colon
C\to X$ a holomorphic map 
 such that the image of the generic point of $C$ lies in
$U\setminus |E|$. Write $V=\varphi^{-1}(U)$.   We assume that the hypotheses of
  Theorem \ref{thm:1} are satisfied, in particular we know that the height jump divisor $J_{\varphi,\ov D}=\varphi^*[\ov D, X]
  - [\varphi|_V^*\ov D, C]$ is an effective $\rr$-divisor on $C$.  The result below implies that the \subpol-Lear extension $[\ov D, X]$ of $\ov D$ over $X$ has non-negative degree on each compact connected Riemann surface mapping into $X$ with generic point mapping into $U$.  
\begin{theorem} \label{loc_integrability}  Assume that  $C$ is compact. Then the
inequalities
\[ \deg \varphi^*[\ov D,X] \geq \deg J_{\varphi,\ov D} \geq 0 
\]
are satisfied.
\end{theorem}
\begin{proof}  From Corollary \ref{cor:2} we infer that 
$\deg\,[\varphi|_V^*\overline D,C]=\int_V \varphi|_V^*\,c_1(\ov D)\ge 0$. From this we obtain that $\deg \varphi^*[\ov D,X] \geq \deg \varphi^*[\ov D,X] -  \deg\,[\varphi|_V^*\overline D,C]=\deg J_{\varphi,\ov D}$. The inequality $\deg J_{\varphi,\ov D} \geq 0$ follows from Theorem \ref{thm:1}.
\end{proof}
Let $g \geq 2$ be an integer. As is explained in \cite[Section~14]{hain_normal}, in the case of the normal function $\nu$ on $\mm_g$ associated to the Ceresa cycle the inequality $\deg
\varphi^*[\nu^* \ov \BB,\ov \mm_g] \geq 0$ that we may deduce from Theorem \ref{loc_integrability} leads to interesting slope
inequalities for compact Riemann surfaces $C$ mapping into $\ov \mm_g$. We recall the matter here briefly.
First of all, it is proved by Hain and D. Reed \cite[Theorem~1.3]{hr} (see also \cite[Theorem~10.1]{hain_normal}) that the Lear
extension $[\nu^* \ov \BB,\ov \mm_g]$ is equal to the so-called
Moriwaki divisor
\[ M = (8g+4)\lambda_1 - g\delta_0 - \sum_{h=1}^{[g/2]} 4h(g-h) \delta_h  \]
on $\ov \mm_g$. Here, following classical notation, $\lambda_1$ denotes the determinant of the Hodge bundle on $\ov \mm_g$, $\delta_0$ is the class of the boundary divisor $\Delta_0$ whose generic point is an irreducible stable curve with one singular point, and $\delta_h$ for $1 \leq h \leq [g/2]$ is the class of the boundary divisor $\Delta_h$ whose generic point is a stable curve consisting of two smooth components, one of genus $h$, and one of genus $g-h$, joined at one point. 

Now, A. Moriwaki has shown \cite{mo} that the divisor $M$ has
non-negative degree on all complete curves in $\ov\mm_g$ not contained
in the boundary. Hain \cite[Proposition~5.4]{hain_normal} notes that $\bb{V}$ extends canonically as a polarized admissible variation of Hodge structure over $\mm_g^c=\ov \mm_g \setminus \Delta_0$, the biextension line bundle $\ov \BB$ extends canonically as a smooth hermitian line bundle over $J(\bb{V})$ over $\mm_g^c$, and $\nu$ extends canonically as an admissible normal function of $J(\bb{V})$ over $\mm_g^c$. In particular (cf. \cite[Theorem~14.1]{hain_normal}), the Moriwaki divisor $M$ has non-negative degree on all complete curves contained in $\mm_g^c$. As already noted in \cite[Section~14]{hain_normal}, Theorem \ref{loc_integrability} combined with \cite[Proposition~5.4]{hain_normal} implies the following joint refinement of Moriwaki's and Hain's results.

\begin{theorem} The Moriwaki divisor $M$ on $\ov \mm_g$ has non-negative degree on all complete curves in $\ov\mm_g$ not contained in $\Delta_0$. 
\end{theorem}
Using the more refined inequality
$\deg \varphi^*[\nu^* \ov \BB,X] \geq \deg J_{\varphi,\nu^* \ov \BB}$ one can obtain better slope inequalities for each map $\varphi \colon C \to \ov \mm_g$  with image not contained in $\Delta_0$ where one can control the height jump divisor $J_{\varphi,\nu^*\ov \BB}$. For example, it was noticed by Hain in \cite[Section~14]{hain_normal} that one has a strictly positive height jump for a small disk mapping generically in $\mm_g$ and passing through a point in the closure of the hyperelliptic locus in $\ov\mm_g$ corresponding to a stable curve which is the union of a genus $h$ stable curve and a genus $g-h-1$ stable curve, intersecting at two points. This example has been extended and analyzed in more detail by Brosnan and Pearlstein, cf. \cite[Theorem~201]{brospearl}.

\subsection{Arakelov line bundles}

We discuss another situation where Theorem \ref{main_intro} can be applied. The context is that of families of compact and connected Riemann surfaces equipped with canonical smooth hermitian line bundles as defined by Arakelov in \cite{ar} using canonical Green's functions. We obtain the necessary Lear extendability from the paper \cite{djfaltings} by the third author, and the necessary semipositivity from the semipositivity of the Poincar\' e bundle on the jacobian, together with an (unsurprising) lemma that states that semipositivity of smooth hermitian line bundles is preserved upon taking Deligne self-pairing.

Let $\var$ and $U$ be smooth complex algebraic varieties and let $\pi \colon \var \to U$ be a proper submersion whose fibers are compact connected Riemann surfaces. Following P. Deligne in \cite{de} one has a canonical holomorphic line bundle $\pair{L,M}$ on $U$ associated to any two holomorphic line bundles $L, M$ on $\var$. The line bundle $\pair{L,M}$  is bi-multiplicative in $L, M$, and its formation is compatible with any base change. If $L, M$ are equipped with smooth hermitian metrics $\|\cdot\|_L$ and $\|\cdot\|_M$, then $\pair{L,M}$ has a canonical induced smooth hermitian metric $\|\cdot\|_{\pair{L,M}}$. With this canonical metric, one has the equality 
\begin{equation} \label{fiber_int} c_1 (\pair{L,M},\|\cdot\|_{\pair{L,M}}) = \int_\pi c_1(L,\|\cdot\|_L) \wedge c_1(M,\|\cdot\|_M)  
\end{equation}
of $(1,1)$-forms on $U$, where $\int_\pi$ denotes integration along the fiber.

\begin{lemma} \label{square} Let $L$ be a holomorphic line  bundle on $\var$ equipped with a smooth hermitian metric $\|\cdot\|$. Assume that the metric $\|\cdot\|$ is semipositive on $\var$. Then the induced metric on the Deligne pairing $\pair{L,L}$ is semipositive on $U$.
\end{lemma}
\begin{proof} Write $\omega = c_1(L,\|\cdot\|)$. Then we have from (\ref{fiber_int}) that
\[ c_1 (\pair{L,L}) = \int_\pi \omega \wedge \omega \, . \]
Let $p \in \var$. We can choose local holomorphic coordinates $u_1,\ldots,u_n,z$ around $p$
with $u_1,\ldots,u_n$ local holomorphic coordinates around $\pi(p)$ on $U$, and
with $z$ a holomorphic coordinate in the fiber. Using the summation convention we can write
\[ \omega = i \left( A_{jk} \, \d u_j \, \d \ov u_k + \d u_j \, b_j^t \, \d \ov z + \d z \, \ov b_j \, \d \ov u_j + c \, \d z \, \d \ov z   \right) \]
with
\[ \Omega = \mat{ A & \ov b \\ b^t & c } \]
a positive semidefinite hermitian matrix. We find
\[ \begin{split} \omega \wedge \omega & = -2 \left( A_{jk}c \, \d u_j \, \d \ov u_k \, \d z \, \d \ov z - \ov b_j \, b_k^t \, \d u_j \, \d \ov u_k \, \d z \, \d \ov z \right) \\
& = 2\,i \, \left( A_{jk}c - \ov b_j\, b_k^t \right) \, \d u_j\, \d \ov u_k \,i\, \d z \, \d \ov z \, . 
\end{split} \]
As $i\, \d z \,\d \ov z$ is semipositive, using a partition of unity we see that it suffices to show that the $(1,1)$-form
\[ i \, \left( A_{jk}c - \ov b_j \, b_k^t \right) \, \d u_j \, \d \ov u_k\]
is semipositive locally on $U$. We are done once we show that the hermitian matrix
\[ A c - \ov b b^t \]
is positive semidefinite. For this it suffices to show that for all $x \in \cc^n$ the inequality
\begin{equation} \label{required} x^t Ac \ov x - x^t \ov b b^t \ov x \geq 0 
\end{equation}
holds.
From the fact that $\Omega\geq 0$ we obtain that for all $x \in \cc^n, y \in \cc$ we have
\begin{equation} \label{whatweknow} x^t A \ov x + x^t \ov b \ov y + y b^t \ov x + y c \ov y \geq 0 \, . 
\end{equation}
Note that both $A, c \geq 0$.
If $c=0$ then it easily follows that $b=0$ and the required inequality (\ref{required}) holds trivially. If $c>0$ we take $y = -x^t \ov b c^{-1}$ in (\ref{whatweknow}) and (\ref{required}) follows again.   
\end{proof}

Let $\omar$ denote the relative dualizing sheaf of $\var \to U$, equipped with the Arakelov metric from \cite{ar}. Assume that the fibers of $\var \to U$ have positive genus $g$. Let $J(\bb{H}) \to U$ denote the jacobian family associated to $\var \to U$, and denote by $\ov \BB$ the Poincar\' e bundle on $J(\bb{H})$, equipped with its canonical smooth hermitian metric.  
\begin{proposition}   \label{with_section} Assume that the map $\var \to U$ has a holomorphic section $P \colon U \to \var$. Then the smooth hermitian line bundle $P^*\omar$ is semipositive on $U$. 
\end{proposition}
\begin{proof} We abbreviate $J=J(\bb{H})$. Let $\delta_P \colon \var \to J$ be the Abel-Jacobi map associated to the section $P$, i.e. the $U$-map $\var \to J$ that for all $u \in U$ sends a point $x \in \var_u$ to the divisor class of $P_u-x$ in $J_u$. Let $\kappa \colon U \to J$ denote the holomorphic map that sends $u \in U$ to the divisor class of $(2g-2)P_u - \omega_{\var_u}$ in $J_u$.  As is shown in \cite[Proposition~8.5]{djfaltings} we have a canonical isometry
\begin{equation} \label{selfintersectionpoint} \left( P^* \omar \right)^{\otimes 4g^2} \isom \pair{\delta_P^*\ov \BB,\delta_P^*\ov \BB} \otimes \kappa^*\ov \BB 
\end{equation}
of smooth hermitian line bundles on $U$. From Theorem \ref{semipositive} we deduce that $\delta_P^*\ov \BB$ is semipositive on $\var$, and that $\kappa^*\ov\BB$ is semipositive on $U$. It follows from Lemma \ref{square} that $\pair{\delta_P^*\ov\BB,\delta_P^*\ov\BB}$ is semipositive on $U$. The result follows.
\end{proof}
\begin{corollary} \label{arakmetric_semipos} The smooth hermitian line bundle $\omar$ is semipositive on $\var$, and the Deligne pairing $\pair{\omar,\omar}$ is semipositive on $U$.
\end{corollary}
\begin{proof} Apply Proposition \ref{with_section} to the first projection $\var \times_U \var \to \var$ and the diagonal section $\Delta \colon \var \to \var \times_U \var$ to obtain the semipositivity of $\omar$ on $\var$. Then apply again Lemma \ref{square} to obtain the semipositivity of $\pair{\omar,\omar}$ on $U$.
\end{proof}
Now we assume a smooth complex variety $X \supset U$ is given with boundary $E=X \setminus U$ a reduced normal crossings divisor. The next result is then part of \cite[Theorem~1.3]{djfaltings}.
\begin{theorem} \label{arakmetric_learexists} Assume that the family $\var \to U$ can be extended into a semistable curve $Y \to X$. Then the Deligne pairing $\pair{\omar,\omar}$ has a Lear extension over $X$. If we assume moreover that $\var \to U$ has a holomorphic section $P \colon U \to \var$, then the smooth hermitian line bundle $P^*\omar$ also has a Lear extension over $X$. 
\end{theorem}
Combining Corollary \ref{arakmetric_semipos} and Theorem \ref{arakmetric_learexists} with Theorem \ref{main_intro} we obtain the following result.
\begin{theorem} 
Assume that the family $\var \to U$ can be extended into a semistable curve $Y \to X$. 
Let $C$ be a connected Riemann surface, and  $\varphi \colon C \to X$ a holomorphic map such that the image of the generic point of $C$ lies in $U$. Then the height jump divisor $J_{\varphi,\pair{\omar,\omar}}$ associated to $\varphi$ and $\pair{\omar,\omar}$ on $C$ is effective. If we assume moreover that the map $\var \to U$ has a holomorphic section $P \colon U \to \var$, then also the height jump divisor $J_{\varphi,P^*\omar}$ on $C$ is effective. 
\end{theorem}


\subsection{Concavity of the height jump}

It is natural to investigate how the height jump divisor changes if
one varies the test curve $\varphi$. Let $X$ be a complex manifold of
dimension $n$,
$E$ a reduced normal crossings divisor on $X$, and put $U = X
\setminus |E|$. Let $\ov D=(D,g_D)$ be a metrized $\rr$-divisor on
$U$, and assume as usual that $\ov D$ is semipositive on $U$ and
sp-Lear extendable over $X$. Let $\Delta^0$ denote the open unit disk
in $\bb{C}$. Let $p \in X$ be a point, let $z=(z_1,\ldots,z_n)$ with
$z_i \in \Delta^0$ be local coordinates around $p$ on $X$ such that
$z_i(p)=0$ for $i=1,\ldots,r$,  and the divisor $E$ is given in the
local coordinates $z$ by the equation $z_1 \cdots z_r = 0$.  Here
$0\le r\le n$.

 Let $V \subset \Delta^0$ denote the punctured open disk. Assume that $\varphi|_V^*\ov D$ is Lear extendable over $\Delta^0$ for all holomorphic maps $\varphi \colon \Delta^0 \to X$ such that $\varphi(0)=p$ and such that $\varphi^{-1}U=V$. Let $\ov D_X=(D_X,g_{D_X})$ be an auxiliary metrized $\rr$-divisor for $\ov D$ on $X$, and write 
\[ [\ov D,X] = D_X + \sum_{i \in I} a_i E_i \]
where $a_i \in \rr$ and where $E_i$ are the irreducible components of $E$. As we are only interested in the local situation around $p$ we may assume without loss of generality that $I=\{1,\ldots,r\}$ and that for all $i=1,\ldots,r$ the irreducible component $E_i$ is given locally on $X$ by the equation $z_i=0$. Let $\varphi \colon \Delta^0 \to X$ be a holomorphic map such that $\varphi(0)=p$ and such that $\varphi^{-1}U=V$. Then $\varphi^*\ov D_X$ is an auxiliary metrized $\rr$-divisor for $\varphi|_V^*\ov D$, and we can write
\[ [\varphi|_V^*\ov D,\Delta^0] = \varphi^*D_X + a \cdot[0] \]
where $a \in \rr$. When one varies $\varphi$ it is readily seen that $a$ only depends on the multiplicities $m_i = \ord_0 (\varphi^*(z_i))$ for $i=1,\ldots,r$. Thus we obtain a function $a \colon \zz_{> 0}^r \to \rr$, with associated jump function $J = J_{\ov D,p} \colon \zz_{> 0}^r \to \rr$ given by 
\[ J(m_1,\ldots,m_r) = -a(m_1,\ldots,m_r)+ \sum_{i=1}^r m_i a_i \, . \]

One easily checks that the function $J$ is homogeneous of weight one, and extends to a function on $\zz_{\geq 0}^r$. Indeed, note that setting $m_i$ equal to zero for all $i \in I'\subseteq I$ corresponds to moving $p$ into the locus where $z_i \neq 0$ for all $i \in I'$. In particular, the local height jump function $J$ extends canonically to a function $\ov J \colon \qq_{\geq 0}^r \to \rr$. 

In \cite[Theorem~5.37]{pearldiff} Pearlstein shows that the local height jumps for the pullback of a biextension line bundle along an admissible normal function as we discussed above are given by a rational function in $\qq(x_1,\ldots,x_r)$. In general, one expects that because of the semipositivity of the metric involved, the function $ \ov J \colon \qq_{\geq 0}^r \to \rr$ should have good concavity properties. Under mild assumptions, this is indeed true, as we show in the next proposition. 
\begin{proposition}  \label{concave} Assume that for all $s \in
  \zz_{>0}$ and for all holomorphic maps $\varphi \colon (\Delta^0)^s
  \to X$ such that 
  $\varphi^{-1}|E|\subset (\Delta^0)^s\setminus V^s$, the metrized
  $\rr$-divisor $\varphi|_{V^s}^*\ov D$ is Lear extendable over
  $(\Delta^0)^s$. Then the extended local height jump function $\ov J
  \colon \qq_{\geq 0}^r \to \rr$ at $p \in X$ as defined above is
  $\qq$-concave.
\end{proposition}
\begin{proof} Fix $s \in \zz_{>0}$ and fix $r$-tuples $v_j=(m_{1,j},\ldots,m_{r,j}) \in \qq^r_{\geq 0}$ for $j=1,\ldots,s$. Let $\lambda_1,\ldots,\lambda_s$ be scalars in $\qq_{> 0}$. We would like to show the inequality
\begin{equation} \label{ineq} \ov J(\sum_{j=1}^s \lambda_j v_j) \geq \sum_{j=1}^s \lambda_j \ov J(v_j) 
\end{equation}
in $\rr$. By the homogeneity of $\ov J$, we may assume $\lambda_j \in \zz_{>0}$ for all $j=1,\ldots,s$, $m_{i,j} \in \zz_{\geq 0}$ for all $i=1,\ldots,r$ and $j=1,\ldots,s$. Write $Y=(\Delta^0)^s$. Consider then the map $\varphi \colon Y \to X$ given in the local holomorphic coordinates $z=(z_1,\ldots,z_n)$ by the assignment 
\[ (y_1,\ldots,y_s) \mapsto (u\cdot y_1^{m_{1,1}}\cdots
  y_s^{m_{1,s}},\ldots,u\cdot y_1^{m_{r,1}}\cdots
  y_s^{m_{r,s}},z_{r+1}(p),\ldots,z_n(p)) \, . \]
where $u\in \rr_{>0}$ is small enough so the image of $Y$ is contained
in the coordinate neighborhood of $p$. 
Also consider the map $\chi \colon \Delta^0 \to Y$ given by $t \mapsto (t^{\lambda_1},\ldots,t^{\lambda_s})$.  We have
\[ \varphi \circ \chi \colon \Delta^0 \to X \, , \quad t \mapsto
  (u\cdot t^{\sum_{j=1}^s \lambda_j m_{1,j}}, \ldots, u\cdot
  t^{\sum_{j=1}^s \lambda_j m_{r,j}},z_{r+1}(p),\ldots,z_n(p)) \]
and we conclude that the height jump divisor $J_{\varphi \circ
  \chi,\ov D}$  is equal to $J(\sum_{j=1}^s \lambda_j v_j)\,[0]$  in
$\Car(\Delta^0)_\rr$. Similarly, if we let $\eta_j \colon \Delta^0 \to
Y$ for $j=1,\ldots,s$ be the map that sends $t$ to $(1/2,\ldots,1/2,
t,1/2,\ldots,1/2)$, with the $t$ at the $j$-th spot, then the height jump divisor for $\varphi \circ \eta_j$  is given by $J(v_j)\,[0] $ in $\Car(\Delta^0)_\rr$. By assumption $\varphi|_{V^s}^*\ov D$ has a Lear extension over $Y$.
Writing $F_1,\ldots,F_s$ for the coordinate divisors in $Y$, the
definition of $\ov J$ gives  the equality 
\[   \varphi^*[\ov D,X] - [\varphi|_{V^s}^*\ov D,Y] = \sum_{j=1}^s \ov
  J(v_j) F_j \]
in $\Car(Y)_\rr$. Applying Theorem \ref{thm:1} to the map $\chi \colon
\Delta^0 \to Y$ we obtain the inequality
\begin{equation} \label{startineq}   \chi^*[\varphi^* \ov D,Y] - [\chi^*\varphi^* \ov D, \Delta^0] \geq 0 \, .
\end{equation}
Here, for readability, we have left out the subscripts $|_V$ and $|_{V^s}$ from the notation. 
On the other hand we may compute
\begin{equation} \label{chainofeq}
\begin{aligned} \chi^*[\varphi^* \ov D,Y] - [\chi^*\varphi^* \ov D,
  \Delta^0]  & = \chi^*\left(-  \sum_{j=1}^s \ov J(v_j) F_j + \varphi^* [\ov D,X]  \right) -[(\varphi \circ\chi)^* \ov D, \Delta^0] \\
& = -\sum_{j=1}^s \lambda_j \ov J(v_j)\,[0] + (\varphi \circ\chi)^*[\ov D,X] -[(\varphi \circ\chi)^* \ov D, \Delta^0] \\
&= -\sum_{j=1}^s \lambda_j \ov J(v_j)\,[0] + J_{\varphi \circ \chi,\ov D} \\
& =  -\sum_{j=1}^s \lambda_j \ov J(v_j)\,[0]  + \ov J(\sum_{j=1}^s \lambda_j v_j)\,[0] \, .
\end{aligned} \end{equation}
Combining (\ref{startineq}) and (\ref{chainofeq}) we obtain (\ref{ineq}) as required.
\end{proof}

Note that $V^s$ has a normal crossings boundary divisor in $(\Delta^0)^s$. Therefore, combining Proposition \ref{concave} with Theorem \ref{lear_exists} and Pearlstein's result \cite[Theorem~5.37]{pearldiff}  we obtain the following.
\begin{theorem} Let $\bb{V}$ be an admissible polarized variation of Hodge structures of weight $-1$ on $U$, let $J(\bb{V})\to U$ denote the associated  intermediate jacobian fibration, and let  $\nu \colon U \to J(\bb{V})$ be an admissible normal function. Let $\ov \BB$ be the biextension line bundle on $J(\bb{V})$ endowed with its canonical metric. Then the height jump function $\ov J \colon \qq_{\geq 0}^r \to \rr$ at $p \in X$
is given by a $\qq$-concave rational function in $\qq(x_1,\ldots,x_r)$.
\end{theorem}

\subsection*{Acknowledgements} We would like to thank R. Hain,
G. Pearlstein and P. Brosnan  for
helpful discussions and pointers to the literature.

\end{document}